\documentclass[]{article}

\newcommand{\un}{\text{ $\underline{\triangleright}$ }}
\newcommand{\ovr}{\text{ $\overline{\triangleright}$ }}
\usepackage[none]{hyphenat}
\usepackage{xcolor, url,booktabs}
\usepackage{array,amsmath,amsfonts,graphicx,amsthm,geometry,setspace,subcaption,overpic,booktabs,tikz-cd}
\usepackage{etoolbox}

\theoremstyle{definition}
\newtheorem{theorem}{Theorem}[subsection]
\newtheorem{example}{Example}[subsection]
\newtheorem{definition}{Definition}[subsection]
\newtheorem{proposition}{Proposition}[subsection]

\newtheorem{remark}{Remark}[subsection]
\newtheorem{corollary}{Corollary}

\title{Biquandle Virtual Brackets and Virtual Knotoids}

\author{Neslihan G\"ug\"umc\"u, Hamdi Kayaslan}
\begin{document}
	\maketitle
    \onehalfspacing
	\begin{abstract}
	In this paper, we introduce invariants of virtual knotoids based on biquandles and biquandle virtual brackets. We show that one of these invariants, namely biquandle virtual bracket matrix, is a proper enhancement of the other invariants introduced in this paper.
	\end{abstract}

	\section{Introduction}
	\paragraph{}

    Virtual knot theory, introduced by L.H. Kauffman \cite{kauffman2012introduction} studies knots and links in thickened surfaces up to homeomorphisms of thickened surfaces plus handle stabilization/destabilization (adding/removing hollow handles). Many classical knot invariants have been extended to virtual knots including invariants based on biquandle colorings of oriented knots and links. See \cite{nelson2017quantum, nelson2019biquandle}.  

   Knotoids are open curves immersed in surfaces having two endpoints that generalize classical knots, and are proven to provide a natural modeling of entangled structures such as proteins and polymers. The theory of knotoids was introduced by V. Turaev \cite{turaev2012knotoids} and has been extensively studied by many researchers. See \cite{ goundaroulis2017studies,goundaroulis2017topological,gugumcu2017new, adams2019knots, gugumcu2021parity,gugumcu2021quantum, gugumcu2022invariants,moltmaker2022framed}.
   
   It is natural to examine knotoids in the context of virtual knot theory. In \cite{gugumcu2017new}, the authors introduced knotoids in surfaces taken up to handle stabilization. Like in the case of virtual knots, it was shown in \cite{gugumcu2017new} that the theory of knotoids in surfaces up to handle stabilization/destabilization is equivalent to the theory of knotoids in $S^2$ with classical and \textit{virtual} crossings that are considered up to Reidemeister moves plus the detour move (generalized Reidemeister moves). With this correspondence, the theory of surface knotoids is referred as virtual knotoid theory. Virtual knotoids were classified by A. Bartholomew up to seven crossings in \cite{virk2}. Recently in \cite{gugumcu2025virtual}, we were able to give a three-dimensional interpretation of virtual knotoids as virtual arcs embedded in thickened surfaces with two endpoints attached on two specified line segments. This interpretation proves that classical knotoid theory embeds into virtual knotoid theory properly. The virtual knotoid theory has not been extensively explored yet, and constructing invariants for virtual knotoids is important in this manner.

        A biquandle $X$ is an algebraic structure with two binary operations satisfying three axioms motivated by Reidemeister moves of classical links. Biquandle brackets defined over a commutative and unitary ring $R$ were introduced by S. Nelson \cite{nelson2017quantum} for $X$- colored knots and links in analogy with the Kauffman bracket. In the Kauffman bracket case, $X=\{1\}$ and $R\,$ is chosen to be the polynomial ring over integers whereas in biquandle bracket case, $X$ can be any finite biquandle, and $R$ can be chosen as any commutative and unitary ring. In \cite{nelson2019biquandle}, this invariant is generalized to virtual knots and links as \textit{biquandle virtual brackets}, with a modification in the skein relation of biquandle brackets where virtual crossings are considered as a smoothing type. 

Later in \cite{gugumcu2019biquandle}, the \textit{biquandle counting invariant} of knotoids is enhanced to the \textit{biquandle counting matrix} by fixing the biquandle colors on the tail and head semi-arcs of a knotoid diagram. In \cite{gugumcu2021biquandle}, biquandle brackets are utilized to obtain the \textit{biquandle bracket matrix} invariant for knotoids that enhances the biquandle counting matrix invariant for knotoids. 

In this paper, we utilize biquandle virtual brackets to enhance the biquandle counting invariant and the biquandle counting matrix and also generalize biquandle bracket matrix for virtual knotoids. We show and compare the strength of our invariants presented with explicit calculations in the case where the ground ring $R$ is chosen to be a number ring.
    
    The paper is organized as follows. In Section \ref{knotoid}, we review the basics of classical and virtual knotoid theories. We review the biquandles in Section \ref{biq}, and study invariants arising from biquandle colorings of virtual knotoids in Section \ref{coloringinv}. In Section \ref{bvbs}, we recall the definition of a biquandle virtual bracket, and use it to introduce new invariants of virtual knotoids as an enhancement of invariants based on biquandle colorings. In this section, we also provide examples that demonstrate the classification abilities of our invariants by using Bartholomew's virtual knotoids table. 

	\section{Knotoids}
    \label{knotoid}
    \label{prelims}
    \paragraph{} In this section, we recall some fundamental notions of knotoids and virtual knotoids.
    
    \begin{definition}\cite{turaev2012knotoids}
        A \textit{knotoid diagram} on a surface $\,\Sigma\,$ is an immersed open curve with two endpoints and a finite number of (transversal) self-intersections. Each self-intersection is endowed with over/under passage information and called a \textit{crossing}. The endpoints of a knotoid diagram are specifically called \textit{tail} and \textit{head} of the diagram. A knotoid diagram is assumed to be oriented from its tail to its head. A \textit{trivial knotoid diagram} is a knotoid diagram with no crossings.
    \end{definition}

    \paragraph{} Knotoid diagrams are considered up to \textit{Reidemeister moves} that are shown in Figure \ref{crm}. These moves are performed away from the endpoints. There are two forbidden moves for knotoid diagrams that consist of pulling/pushing an arc adjacent to an endpoint under/over a transversal arc. These moves are forbidden because any knotoid diagram in any surface can be turned into a trivial knotoid diagram via a sequence of these moves.
    
    \begin{figure}[h!]
    \centering
        \begin{subfigure}{0.4\textwidth}
        \centering
            \includegraphics[width=50 mm]{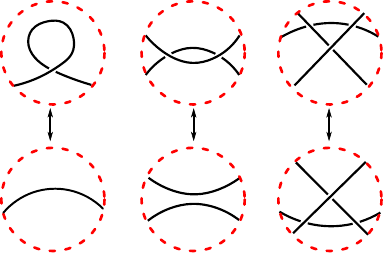}
            \caption{Classical Reidemeister moves.}
            \label{crm}
        \end{subfigure}
        \begin{subfigure}{0.5\textwidth}
        \centering
            \includegraphics[width=50 mm]{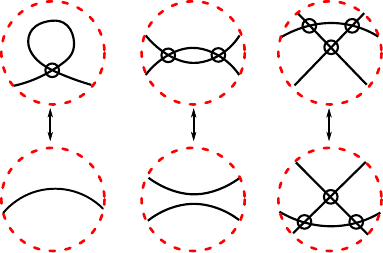}
            \vspace{0.01 cm}
            \caption{Virtual Reidemeister moves.}
            \label{vrm}
        \end{subfigure}
        \begin{subfigure}{0.5\textwidth}
        \centering
        \vspace{0.2cm}
            \includegraphics[width=36 mm]{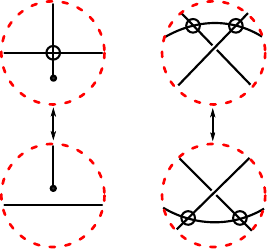}
            \vspace{0.01 cm}
            \caption{Mixed Reidemeister moves.}
            \label{mrm}
        \end{subfigure}
        \caption{Extended Reidemeister moves.}
        \label{erm}
    \end{figure}

    \begin{definition}
        A \textit{knotoid} is an equivalence class of knotoid diagrams in a surface with respect to the equivalence relation induced by Reidemeister moves and isotopy of the surface. The set of knotoids on a surface $\Sigma$ is denoted by $\mathcal{K}(\Sigma)$. A knotoid in $S^2$ is called a \textit{classical} knotoid.
    \end{definition}

     \paragraph{} Classical knotoids were extended to virtual knotoids in \cite{gugumcu2017new}. Let us make a brief review of this extension.

    \begin{definition}
        A \textit{virtual knotoid diagram} is a knotoid diagram in $S^2$ with a finite number of classical and virtual crossings. A virtual crossing is indicated by a circle around the crossing point without under or over information. See Figure \ref{virtk1} for an example of a virtual knotoid diagram with one virtual crossing and two classical crossings.

       \begin{figure}[h!]
		\centering
        \includegraphics[width=30 mm]{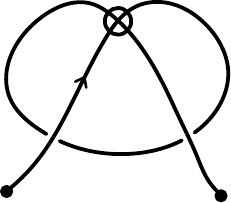}
			\caption{A virtual knotoid diagram.}
			\label{virtk1}
	\end{figure}

  Two virtual knotoid diagrams are considered to be \textit{equal} if one can be turned into another via a sequence of \textit{extended Reidemeister moves} shown in Figure \ref{erm} and isotopy of $S^2$. Equivalence classes of virtual knotoid diagrams with respect to this equivalence relation are called \textit{virtual knotoids}. 
    \end{definition}

It was shown in \cite{gugumcu2017new} that virtual knotoids can be interpreted as knotoids in higher genus closed, connected, orientable (c.c.o) surfaces that are considered up to \textit{stable equivalence}, that is induced by Reidemeister moves in the surfaces, isotopy of the surfaces and the addition/subtraction of empty handles in the complement of the diagrams. This connection between virtual knotoids and knotoids in higher genus c.c.o surfaces was established via \textit{abstract knotoid diagrams}. See \cite{gugumcu2017new} for details on this connection. Later in \cite{gugumcu2025virtual}, \textit{virtual arcs} were introduced. A \textit{virtual arc} is an open-ended curve embedded in a thickened c.c.o surface, whose endpoints are considered to be attached on a pair of line segments in the thickened surface. See Figure \ref{tthick}. 
It was shown in \cite{gugumcu2025virtual} that virtual knotoids are in one-to-one correspondence with virtual arcs in thickened surfaces, considered up to \textit{virtual arc isotopy} that transforms virtual arcs to each other via stable equivalences that are realized in the complement of the pair of line segments to which the endpoints of virtual arcs are attached.  The readers are referred to \cite{gugumcu2025virtual, chmutov2024thistlethwaite} for more on the subject. We state the following theorem that was deduced from these connections in \cite{gugumcu2025virtual}.

\begin{figure}[h!]
        \centering
        \includegraphics[width=0.33\linewidth]{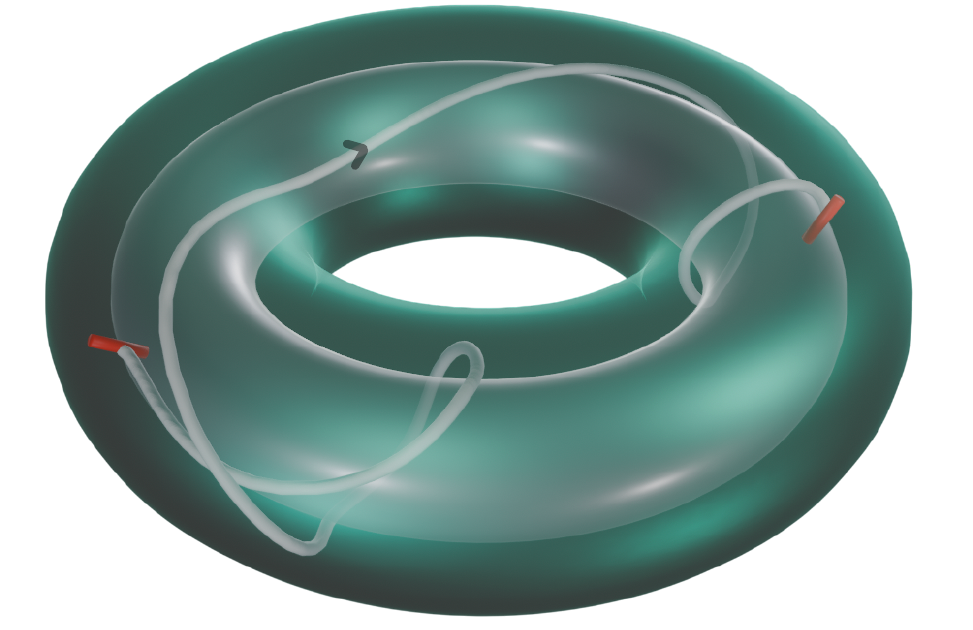}
        \caption{A virtual arc in a thickened torus whose endpoints are attached to two line segments (colored in red).}
        \label{tthick}
    \end{figure}

\begin{theorem}
Two classical knotoid diagrams $K_1$, $K_2$ are equal to each other via a sequence of extended Reidemeister moves if and only if they are related to one another via a sequence of only classical Reidemeister moves. 
\end{theorem}

As a corollary, we can study classical knotoids in the framework of virtual knotoids without altering their topological types (e.g. equivalence classes). Precisely, a classical knotoid can be viewed as a virtual knotoid that can be represented by a classical knotoid diagram.

    \paragraph{} The connected sum operation of classical knotoids, defined in \cite{turaev2012knotoids}, extends to virtual knotoids as follows.

   \begin{definition}
       
  Let $K_i$ be a virtual knotoid in $S^2$  having endpoints  $t_i, h_i$ for $i=1,2$. Let $U$ and $V$ be disk neighborhoods of $h_1$ and $t_2$ that intersect $K_1$ and $K_2$ along the radii of the disks, respectively. The \textit{product} of virtual knotoids $K_1$ and $K_2$, denoted by $K_1K_2$, is the virtual knotoid resulted from gluing $S^2-\text{Int}(U)$ and $S^2-\text{Int}(V)$ along a homeomorphism $\phi:\partial(U)\rightarrow \partial(V)$ such that the intersection point $K_1\cap \partial(U)$ is mapped to the intersection point $K_2 \cap \partial(V)$. $K_1 K_2$ is a virtual knotoid in $S^2$ having tail endpoint $t_1$ and head endpoint $h_2$. See Figure \ref{consum} for an illustration.

   \end{definition}

    \begin{figure}[h!]
        \centering
        \includegraphics[width=0.45\linewidth]{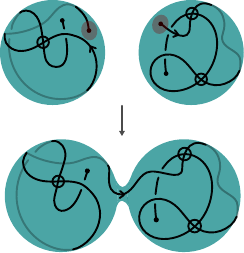}
        \caption{Product of two virtual knotoid diagrams in $S^2$.}
        \label{consum}
    \end{figure}

   Note that the set of all virtual knotoids endowed with the connected sum operation is a monoid.

	\section{Biquandles and Biquandle Coloring Invariants}
    \label{biqinvariants}
    This section is a brief review of biquandles and invariants of knotoids based on biquandles. We first introduce biquandles and provide some examples.
    \subsection{Biquandles}
    \label{biqs}
    We begin with the definition of a biquandle.
	\begin{definition}
		Let $X$ be a set with two operations $\un,\ovr:X\times X\rightarrow X$. If for all $x,y,z\in X$,
		\begin{itemize}
			\item[(1)] $x\un x=x\ovr x$,
			\item[(2)] the following maps are invertible,
			\begin{align*}
				&\alpha_y:X\rightarrow X, \qquad \beta_y:X\rightarrow X,\qquad S:X\text{ x } X\rightarrow X\text{ x }X,\\
				&\quad x\mapsto x\ovr y\qquad x\mapsto x\un y\qquad\quad (x,y)\mapsto (y\ovr x, x\un y)
			\end{align*}
			\item[(3)] and the exchange laws hold, \begin{align*}
				(x\un y)\un(z\un y)=(x\un z)\un(y\ovr z),\\
				(x\un y)\ovr(z\un y)=(x\ovr z)\un(y\ovr z),\\
				(x\ovr y)\ovr(z\ovr y)=(x\ovr z)\ovr(y\un z),
			\end{align*}
		\end{itemize}
		then, $(X,\un,\ovr)$ is called a \textit{biquandle}. A biquandle $(X,\un,\ovr)$ is usually denoted by $X$.
        \label{bq}
    \end{definition}

    \begin{example}
    \label{biqex}
        Let $R$ be a module over the ring $\mathbb{Z}[t^{\pm1},s^{\pm1}]$ and define the operations on $R$ for all $x,y\in R$ as
        \[ x\un y=tx+(r-t)y\quad\text{ and }\quad x\ovr y=rx.\]
        $(R,\un,\ovr)$ determines a biquandle structure that is called an \textit{Alexander biquandle.} 
    \end{example}

 Specifically, let $R=\mathbb{Z}_5$ and $t=2$, $r=4$. Consider the Alexander biquandle on $R$. The operations on $R$ become 
        \[ x\un y= 2(x+y),\quad x\ovr y= 4x.\] 
        \label{biq}

    \begin{definition}
        
Let $X=\{x_1,x_2,...,x_n\}$ be a biquandle. Then the augmented $n\times 2n$ matrix $A=\{a_{i,j}\}$ where for all $i\in \{1,2,...,n\}$, \[a_{i,j}= k \quad \text{if} \quad x_k=\begin{cases}
        x_i\un x_j, & 1\leq j\leq n\\
        x_i\ovr x_j, & n+1\leq j\leq 2n
    \end{cases}\] is called the \textit{biquandle operation matrix} of $X$.
        \end{definition}
    For example, the operation matrix of the biquandle in Example \ref{biqex} is \[\left[\begin{array}{c c c c c | c c c c c}
             4& 1& 3& 5& 2& 4& 4& 4& 4& 4  \\
             1& 3& 5& 2& 4& 3& 3& 3& 3& 3\\
             3& 5& 2& 4& 2& 2& 2& 2& 2& 2\\
             5& 2& 4& 1& 3& 1& 1& 1& 1& 1\\
             2& 4& 1& 3& 5& 5& 5& 5& 5& 5
        \end{array}\right]\]
        where the equivalence class of 0 is represented by 5 in $\mathbb{Z}_5$.
        
    \begin{remark}
        In our examples containing biquandle structures over $\mathbb{Z}_n$, we represent the equivalence class of 0 by $n$ as it is more convenient for the biquandle operation matrices.
    \end{remark}
    
    The axioms of a biquandle (see Definition \ref{bq}) are motivated by classical Reidemeister moves of virtual knot/knotoid diagrams. To see this, one can consider labeling each semi-arc of an oriented virtual knot or knotoid diagram with an element of a set $X$ that is endowed with two operations $\ovr$ and $\un$ in the following way.
    The semi-arcs adjacent to a positive and negative classical crossing of a knot/knotoid diagram are labeled as in Figure \ref{bqconvention} parts (a) and (b), respectively. The semi-arcs adjacent to a virtual crossing are labeled as in part (c) of Figure \ref{bqconvention}. The labels of the semi-arcs do not change through a virtual crossing. The axioms of a biquandle correspond to relations that we get in order to make this labeling a topological invariant. 
    We refer the reader to \cite{nelson2017quantum} for more details.

    \begin{figure}[h!]
			\centering
			\begin{overpic}[width=0.5 \linewidth]{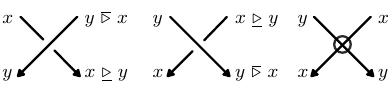}
            \put(10,-2){(a)}
            \put(48,-2){(b)}
            \put(85,-2){(c)}
        \end{overpic}
			\caption{Biquandle labelings around classical and virtual crossings.}
			\label{bqconvention}
		\end{figure}

	\begin{definition}
		Let $K$ be an oriented virtual knotoid diagram whose semi-arcs are labeled by $x_1, x_2, ..., x_m$. The \textit{fundamental biquandle of  K} is the set $\mathcal{B}(K)$ whose elements are congruence classes of finite words generated by the labels of semi-arcs of $K$, up to the biquandle axioms and the relations arising from the crossing labeling conditions in Figure \ref{bqconvention}.

        In Figure \ref{vrandomb}, we give the crossing relations of the fundamental biquandle of the virtual knotoid diagram $K$. The fundamental biquandle of $K$ is given as:

    \[
        \mathcal{B}(K)=\;<a,b,c,d,e\;|\; c\un b=d, b\ovr c=a, d\ovr c=e, c\un d=b>.
    \]
    \begin{figure}[h!]
        \centering
        \includegraphics[width=0.8\linewidth]{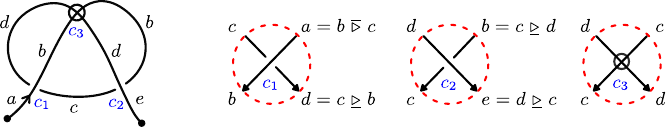}
        \caption{A virtual knotoid diagram and its crossing relations.}
        \label{vrandomb}
    \end{figure}
   \end{definition}
 
	\begin{definition}
		Let X and Y be two biquandles. A map $f:X\rightarrow Y$ is called a \textit{biquandle homomorphism} if for all $x,y\in X$ we have \[f(x\un y)=f(x)\un f(y),\quad f(x\ovr y)=f(x)\ovr f(y).\] The set of biquandle homomorphisms from $X$ to $Y$ is denoted by Hom$(X,Y)$.
	\end{definition}

    \subsection{Biquandle Coloring Invariants of Virtual Knotoids}
    \label{coloringinv}
    
    In this section, we study invariants arising from \textit{colorings (labelings)} of virtual knotoid diagrams by using elements of a finite biquandle. We provide some examples of biquandle colorings for virtual knotoids. 
    
    \begin{definition}
        Let $K$ be a virtual knotoid diagram and $X$ be a finite biquandle. An $X\text{-}coloring$ of $K$ is a labeling of the semi-arcs of $K$ with elements of $X$ such that the labeling around each crossing respects the corresponding relation given in Figure \ref{bqconvention}. The number of $X$-colorings of $K$ is denoted by $\Phi_X^{\mathbb{Z}}(K)$.
    \end{definition}

    \begin{remark}
    If $K$ admits only classical crossings, the first two of the labeling conventions given in Figure \ref{bqconvention} are utilized. 
         In \cite{gugumcu2019biquandle} and \cite{gugumcu2021biquandle}, Nelson and the first author study biquandle coloring invariants of classical knotoids.
    \end{remark}
    
    \begin{remark}
        For our examples in the sequel, we will be using the virtual knotoid table given by Bartholomew in \cite{virk2}. In \cite{virk2}, virtual knotoids are studied with their \textit{labeled peer codes}. We refer the interested reader to \cite{virk2} for more detail on labeled peer codes. In our examples, we first generate virtual knotoid diagrams corresponding to labeled peer codes. The virtual knotoid diagrams are named the same as in \cite{virk2}. We then utilize our Python code for our computations. 
    \end{remark}

        \begin{figure}[h!]
        \begin{subfigure}{0.38\textwidth}
        \centering
            \includegraphics[width=25 mm]{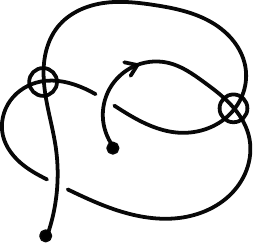}
            \vspace{1 cm}
            \caption{A diagram $K$ of the virtual knotoid 2.1.1.}
            \label{21}
        \end{subfigure}
        \begin{subfigure}{0.5\textwidth}
        \centering
            \includegraphics[width=80 mm]{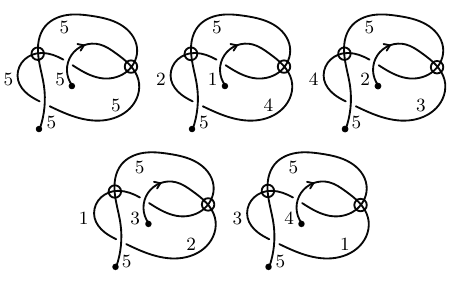}
            \caption{$\mathbb{Z}_5$-colorings of $K$.}
            \label{21color}
        \end{subfigure}
        \caption{}
    \end{figure}

    \begin{example}
        Let $K$ be an oriented diagram of the virtual knotoid 2.1.1 given in Figure \ref{21}, and let $X=\mathbb{Z}_5$ be the Alexander biquandle (see Example \ref{biq}). We order the semi-arcs of $K$ by following the orientation of the diagram, starting from the semi-arc adjacent to the tail of $K$. Then we write out the equations derived from the crossings of $K$, and row reduce the coefficient matrix over $\mathbb{Z}_5,$ and we see that the kernel is generated by (3,2,1,5,5). Then $\Phi_X^{\mathbb{Z}}(K)=5.$ Figure \ref{21color} shows all possible $X$-colorings of $K$.
        \label{bqcolor}
    \end{example}

    Let $X$ be a finite biquandle and $K$ be a virtual knotoid diagram. An $X$-coloring of $K$ can be considered as a biquandle homomorphism $f: \mathcal{B}(K)\rightarrow X$, that is determined by sending each generator of the fundamental biquandle $\mathcal{B}(K)$ to an element of $X$. 
    See Figure \ref{bqqhom} for a local illustration of such mapping. Observe that the images of the generators around the given crossing determine a proper coloring of the semi-arcs at that crossing by the elements of $X$ under a biquandle homomorphism $f:\mathcal{B}(K)\rightarrow X$. 
    
    In fact, there is one-to-one correspondence between $X$-colorings of $K$ and Hom($\mathcal{B}(K),X)$. By this correspondence, we use the elements $f\in$ Hom($\mathcal{B}(K),X$) to denote the $X$-colorings of $K$. We use $K_f$ to denote a virtual knotoid diagram with an $X$-coloring $f$. 

    \begin{figure}[h!]
        \centering
        \includegraphics[width=0.5\linewidth]{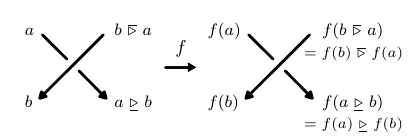}
        \caption{ Fundamental biquandle elements around a crossing and their images under a homomorphism $f$.}
        \label{bqqhom}
    \end{figure}
    \begin{example}
    \label{trivial}
        Let $K$ be a virtual knotoid diagram. Let $1_{\mathcal{B}(K)}$ denote the identity map of the fundamental biquandle $\mathcal{B}(K)$ of $K$. Then, $1_{\mathcal{B}(K)}$ represents the $\mathcal{B}(K)$-coloring of $K$, where each generator of $\mathcal{B}(K)$ colors its corresponding semi-arc in $K$. We call $1_{\mathcal{B}(K)}$ the \textit{trivial fundamental biquandle coloring} of $K$, and denote the trivial fundamental biquandle colored diagram $K$ by $K_{1_{\mathcal{B}}}$.
    \end{example}

    \begin{proposition}
        $\Phi_X^{\mathbb{Z}}(K)$ is an invariant of virtual knotoids.
    \end{proposition}
    \begin{proof}
    The way in which the biquandle axioms are chosen guarantees that an $X$-coloring of a virtual knotoid diagram $K$ corresponds to a unique $X$-coloring of the virtual knotoid diagram $K'$ which is obtained by performing a generalized Reidemeister move on $K$. This implies that the number of $X$-colorings of virtual knotoid diagrams is preserved under generalized Reidemeister moves.
    \end{proof}

    \begin{example}
    \label{z332}
        Consider the biquandle $X=\mathbb{Z}_3$ with the following operation matrix
        \[\left[\begin{array}{c c c | c c c}
            3&2&1&3&3&3\\
            2&1&3&1&1&1\\
            1&3&2&2&2&2\\
        \end{array}\right].\]
        Our Python codes show that the virtual knotoid diagram representing the virtual knotoid 3.1.2, given in Figure \ref{32}, has three $Z_3$-colorings (see these three colorings in Figure \ref{32color}) and the virtual knotoid diagram in Figure \ref{21} has no $Z_3$-colorings. Therefore, the biquandle counting invariant distinguishes the virtual knotoids 3.1.2 and 2.1.1 from each other. 
        \end{example}

    \begin{figure}[h!]
        \begin{subfigure}{0.4\textwidth}
        \centering
            \includegraphics[width=25 mm]{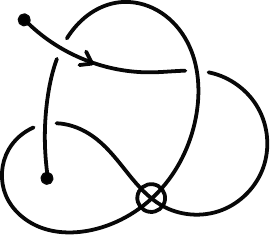}
            \vspace{1 cm}
            \caption{A diagram $K$ of the virtual knotoid 3.1.2.}
            \label{32}
        \end{subfigure}
        \begin{subfigure}{0.5\textwidth}
        \centering
            \includegraphics[width=60 mm]{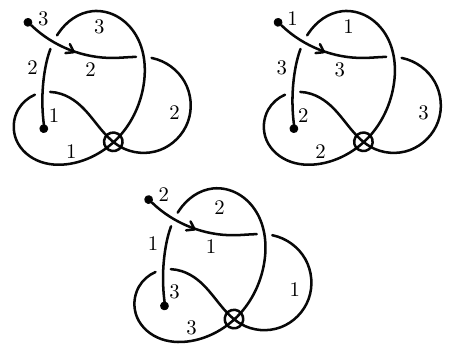}
            \caption{$\mathbb{Z}_3$-colorings of $K$.}
            \label{32color}
        \end{subfigure}
        \caption{Distinguishing two virtual knotoids by biquandle colorings.}
    \end{figure}

    The main difference between knots and knotoids is that knotoids have two special semi-arcs that are the semi-arcs adjacent to the tail and head of the knotoid. 
    In \cite{gugumcu2019biquandle}, the biquandle counting invariant of classical knotoids is strengthened by fixing colors on the tail and head semi-arcs. We generalize this approach to virtual knotoids as follows.
    
	\begin{definition}
		Let $X$ be a finite biquandle, and let $K$ be a virtual knotoid diagram. For any two elements $x_i,x_j\in X$, an \textit{$X_{ij}$-coloring} of $K$ is a particular $X$-coloring of $K$ where the tail semi-arc of $K$ is colored by $x_i$ and the head semi-arc of $K$ is colored by $x_j$. The number of $X_{ij}$-colorings of $K$ is denoted by $\Phi_{X_{ij}}^{\mathbb{Z}}(K)$.
	\end{definition}

    The set of $X_{ij}$-colorings of a virtual knotoid diagram $K$ is in a one-to-one correspondence with the elements $f\in \text{Hom}(\mathcal{B}(K),X)$ where the generators of the tail semi-arc and the head semi-arc are mapped onto $x_i$ and $x_j$, respectively, in $X$. We denote the set of such homomorphisms by $\text{Hom}_{ij}(\mathcal{B}(K),X)$.
    
    \begin{example}
        Consider the virtual knotoid diagram $K$ in Figure \ref{32} and the biquandle $X$ of Example \ref{z332}. Clearly,
        \[\Phi_{X_{ij}}^{\mathbb{Z}}(K)=\begin{cases}
            1,& j=i+1 \;(\text{mod } 3),\\
            0,&\text{otherwise}.
        \end{cases}\]
        \label{xij}
    \end{example}

    \begin{proposition}
		The number of biquandle $X_{ij}$-colorings, $\Phi_{X_{ij}}^{\mathbb{Z}}$, is an invariant of virtual knotoids.
        \label{color}
	\end{proposition}

    It is advantageous to store the $X_{ij}$-coloring number data of a virtual knotoid diagram $K$ in a matrix in the following way.
    
	\begin{definition}
		Let $X$ be a finite biquandle and let $K$ be a virtual knotoid diagram. The $\textit{biquandle counting matrix}\,$ of $K$, denoted by $\mathcal{M}_X(K)$, is a matrix whose $(i,j)$-th entry is the number of $X_{ij}$-colorings of $K$.
	\end{definition}
    
    \begin{example}
        Consider the virtual knotoid diagram $K$ and the biquandle $X$ of Example \ref{xij}. Then,
        \[\mathcal{M}_X(K)=\left[\begin{array}{c c c c c}
            0&1&0\\
            0&0&1\\
            1&0&0
        \end{array}\right].\]
    \end{example}

    \begin{theorem}
        The biquandle counting matrix $\mathcal{M}_X(K)$ is an invariant of virtual knotoids.
    \end{theorem}
    The sum of the entries of $\mathcal{M}_X(K)$ gives the biquandle counting invariant $\Phi_X^{\mathbb{Z}}(K)$ of $K$. Thus, the biquandle counting matrix is an enhancement of the biquandle counting invariant for virtual knotoids. The following example shows that the biquandle counting matrix is strictly stronger than the biquandle counting invariant.
    \begin{example}
    \begin{figure}[h!]
        \begin{subfigure}{0.5\textwidth}
        \centering
            \includegraphics[width=28 mm]{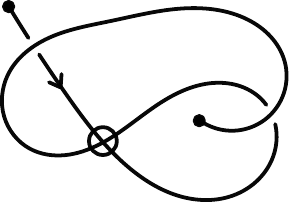}
            \caption{The virtual knotoid diagram 2.1.2.}
            \label{22}
        \end{subfigure}
        \begin{subfigure}{0.5\textwidth}
        \centering
            \includegraphics[width=32 mm]{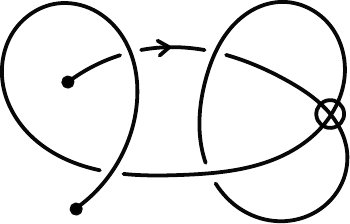}
            \caption{The virtual knotoid diagram 4.1.2.}
            \label{42}
        \end{subfigure}
        \caption{}
        \label{2242}
    \end{figure}
    Let $K_1$ and $K_2$ be the virtual knotoid diagrams given in Figure \ref{22} and \ref{42}, respectively. Consider the biquandle $X=Z_3$ with the operation matrix given in Example \ref{z332}. We compute the biquandle counting matrices of $K_1$ and $K_2$ using our Python code and observe that,

    \[\mathcal{M}_X(K_1)=\left[\begin{array}{c c c c c}
            1&0&0\\
            0&1&0\\
            0&0&1
        \end{array}\right]\neq \left[\begin{array}{c c c c c}
            0&1&0\\
            0&0&1\\
            1&0&0
        \end{array}\right]=\mathcal{M}_X(K_2).\]
    The biquandle counting matrix distinguishes these two virtual knotoid diagrams where the biquandle counting invariant does not, as $\Phi_X^{\mathbb{Z}}(K_1)=3=\Phi_X^{\mathbb{Z}}(K_2)$.
    \end{example}

Note that we have worked on biquandle colorings of prime virtual knotoids in our examples so far. Given any two knotoid diagrams $K_1$, $K_2$, and two $X$-colorings of $K_1$ and $K_2$, respectively. It is clear that the $X$-colorings of $K_1$ and $K_2$ can be extended to a unique $X$-coloring of the product $K_1K_2$ if the head semi-arc of $K_1$ has the same color as the tail semi-arc of $K_2$. The following corollaries follow easily from this observation.

\begin{corollary}
  Let $X=\{x_1,...,x_n\}$ be a biquandle, and $K_1$ and $K_2$ be two virtual knotoid diagrams. We have, 
    \[ \Phi_{X_{ik}}^{\mathbb{Z}}(K_1K_2)=\Phi_{X_{ij}}^{\mathbb{Z}}(K_1) \Phi_{X_{jk}}^{\mathbb{Z}}(K_2),\]
    for all $i,j,k\in \{1,...,n\}.$
\end{corollary}

    \begin{corollary}
        
     (Proven in \cite{gugumcu2021biquandle} for classical knotoids.)
        Let $K_1$ and $K_2$ be two virtual knotoid diagrams. Then for any finite biquandle $X$, the biquandle counting matrix of $K_1K_2$ is the matrix product of the biquandle counting matrices of $K_1$ and $K_2$, i.e.
        \[\mathcal{M}_X(K_1K_2)=\mathcal{M}_X(K_1)\mathcal{M}_X(K_2).\]
    \end{corollary}
	
    \section{Biquandle Virtual Brackets and Virtual Knotoids}
    \label{bvbs}moo

    Biquandle virtual brackets were introduced by Nelson in \cite{nelson2019biquandle} as a generalization of biquandle brackets in \cite{nelson2017quantum} to oriented virtual knots and links.

    A biquandle virtual bracket is defined as follows. 

\begin{definition}\label{defn:biquandlebracket}
		Let $X$ be a biquandle and $R$ be a commutative, unitary ring. A \textit{biquandle} \textit{X-virtual bracket}, denoted by $\beta$, consists of maps $A,B,V,C,D,U:X\text{ x }X\rightarrow R$ and two elements $\delta\in R$, $\omega\in R^{\text{x}}$ such that for all $x,y,z\in X$ the following equations $(1)-(23)$ hold, where for all $x,y\in X$ and $F\in\{A,B,V,C,D,U\}$,  $F(x,y)$ is denoted by $F_{x,y}$, and for convenience, $x\un y$ and $x\ovr y$ are denoted by $x^y$ and $x_y$, respectively. We name the maps $A,B,V,C,D,U$ as \textit{coefficient maps}.
		\begin{align}
			\omega&=&&\hspace{-4.2cm}\delta A_{x,x}+B_{x,x}+V_{x,x},\\
			\omega^{-1}&=&&\hspace{-4.2cm}\delta C_{x,x}+D_{x,x}+U_{x,x},\\
			1&=&& \hspace{-4.2cm}A_{x,y}C_{x,y}+V_{x,y}U_{x,y},\\
			1&=&& \hspace{-4.2cm}B_{x,y}D_{x,y}+V_{x,y}U_{x,y},\\
			0&=&& \hspace{-4.2cm}A_{x,y}U_{x,y}+V_{x,y}C_{x,y},\\
			0&=&& \hspace{-4.2cm}B_{x,y}U_{x,y}+V_{x,y}D_{x,y},\\
			0&=&&\hspace{-4.2cm}\delta B_{x,y}D_{x,y}+A_{x,y}D_{x,y}+B_{x,y}C_{x,y},\\
			0&=&&\hspace{-4.2cm}\delta A_{x,y}C_{x,y}+A_{x,y}D_{x,y}+B_{x,y}C_{x,y},
        \end{align}
        \begin{align}
			A_{x,y} A_{x^y,z_y} A_{y,z}+V_{x,y} A_{x^y,z_y} V_{y,z}&=&&\hspace{-1.4cm}A_{y_x,z_x} A_{x,z} A_{x^z,y^z}+V_{y_x,z_x} A_{x,y} V_{x^z,y^z}\\
			A_{x,y} A_{x^y,z_y} B_{y,z}+B_{x,y} A_{x^y,z_y} A_{y,z}&&&\notag\\
			+\delta B_{x,y} A_{x^y,z_y} B_{y,z}+B_{x,y} A_{x^y,z_y} V_{y,z}&&&\notag\\
			+B_{x,y} B_{x^y,z_y} B_{y,z}+A_{x,y} V_{x^y,z_y} B_{y,z}&&&\notag\\
			+V_{x,y} A_{x^y,z_y} B_{y,z}&=&&\hspace{-1.4cm}A_{y _x,z_x} B_{x,z} B_{x^z,y^z}\\
			A_{x,y} B_{x^y,z_y} B_{y,z}&=&&\hspace{-1.4cm}A_{y_x,z_x} A_{x,z} B_{x^z,y^z}+B_{y_x,z_x} A_{x,z} A_{x^z,y^z}\notag\\
			&&&\hspace{-1.4cm}+\delta B_{y_x,z_x} A_{x,z} B_{x^z,y^z}+B_{y_x,z_x} A_{x,z} V_{x^z,y^z}\notag\\
			&&&\hspace{-1.4cm}+B_{y_x,z_x} B_{x,z} B_{x^z,y^z}+B_{y_x,z_x} V_{x,z} B_{x^z,y^z}\\
			A_{x,y} V_{x^y,z_y} A_{y,z}&=&&\hspace{-1.4cm}A_{y_x,z_x} A_{x,z} V_{x^z,y^z}+V_{y_x,z_x} A_{x,z} A_{x^z,y^z}\\
			A_{x,y} A_{x^y,z_y} V_{y,z}+V_{x,y} A_{x^y,z_y} A_{y,z}&=&&\hspace{-1.4cm}A_{y_x,z_x} V_{x,z} A_{x^z,y^z}\\
			B_{x,y} B_{x^y,z_y} A_{y,z}+B_{x,y} V_{x^y,z_y} V_{y,z}&=&&\hspace{-1.4cm}A_{y_x,z_x} B_{x,z} B_{x^z,y^z}+V_{y_x,z_x} V_{x,z} B_{x^z,y^z}\\
			A_{x,y} B_{x^y,z_y} B_{y,z}+V_{x,y} V_{x^y,z_y} B_{y,z}&=&&\hspace{-1.4cm}B_{y_x,z_x} B_{x,z} A_{x^z,y^z}+B_{y_x,z_x} V_{x,z} V_{x^z,y^z}\\
			B_{x,y} B_{x^y,z_y} V_{y,z}+B_{x,y} V_{x^y,z_y} A_{y,z}&=&&\hspace{-1.4cm}A_{y_x,z_x} B_{x,z} V_{x^z,y^z}\\
			A_{x,y} B_{x^y,z_y} V_{y,z}&=&&\hspace{-1.4cm}B_{y_x,z_x} B_{x,z} V_{x^z,y^z}+B_{y_x,z_x} V_{x,z} A_{x^z,y^z}\\
		V_{x,y} B_{x^y,z_y} A_{y,z}&=&&\hspace{-1.4cm} A_{y_x,z_x} V_{x,z} B_{x^z,y^z}+V_{y_x,z_x} B_{x,z} B_{x^z,y^z}\\
			A_{x,y} V_{x^y,z_y} B_{y,z}+V_{x,y} B_{x^y,z_y} B_{y,z}&=&&\hspace{-1.4cm}V_{y_x,z_x} B_{x,z} A_{x^z,y^z}\\
			V_{x,y} V_{x^y,z_y} A_{y,z}&=&&\hspace{-1.4cm}A_{y_x,z_x} V_{x,z} V_{x^z,y^z}\\
			A_{x,y} V_{x^y,z_y} V_{y,z}&=&&\hspace{-1.4cm}V_{y_x,z_x} V_{x,z} A_{x^z,y^z}\\
			V_{x,y} B_{x^y,z_y} V_{y,z}&=&&\hspace{-1.4cm}V_{y_x,z_x} B_{x,z} V_{x^z,y^z}\\
			V_{x,y} V_{x^y,z_y} V_{y,z}&=&&\hspace{-1.4cm}V_{y_x,z_x} V_{x,z} V_{x^z,y^z}
		\end{align}
	\end{definition} 
    
   A biquandle virtual bracket invariant of biquandle colored virtual knots is defined in analogy with the Kauffman bracket by Nelson in $\cite{nelson2019biquandle}$. We utilize/mimic this idea to construct biquandle virtual bracket invariant of virtual knotoids as follows.
    
    Let $X$ be a biquandle, $R$ be a commutative ring with unit element, and $K_f$ be an oriented virtual knotoid diagram with $n$ crossings that is endowed with an $X$-coloring $f$. By first ignoring the orientation on $K_f$, each classical crossing of $K_f$ is replaced by either vertical or horizontal strands as in the Kauffman bracket, or by a virtual crossing as shown in Figure \ref{smooth}. Each of these replacements is called a \textit{smoothing} of the crossing. We specifically call the virtual crossing replacement a \textit{virtual smoothing}. 
    
    \begin{figure}[h!]
       \centering
        \begin{overpic}[width=0.35\linewidth]{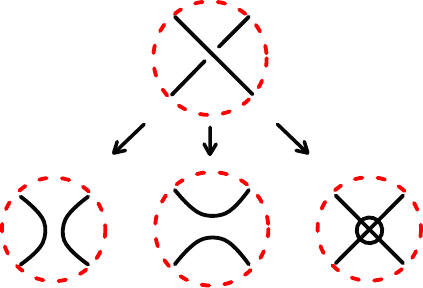}
        \put(8,-8){(a)}
        \put(46,-8){(b)}
        \put(83,-8){(c)}
        \end{overpic}
        \vspace{0.4 cm}
        \caption{Vertical, horizontal and virtual smoothings of a classical crossing.}
        \label{smooth}
    \end{figure}

    By smoothing out each crossing of $K_f$ in all possible ways, we obtain a collection of a union of a number of closed curves and a single open-ended curve that may contain virtual crossings. Each member of this collection is called a \textit{state} of $K_f$. It is clear that $K_f$ has $3^n$ states. A biquandle colored virtual knotoid diagram $K_f$ with two classical crossings $c_1$ and $c_2$, and two states (a) and (b) of $K_f$, are given in Figure \ref{states}. In (a), the smoothings of $c_1$ and $c_2$ are both horizontal. In (b), the smoothing of $c_1$ is virtual, and the smoothing of $c_2$ is horizontal. Note that the state (a) has two components, whereas (b) has a single component.

     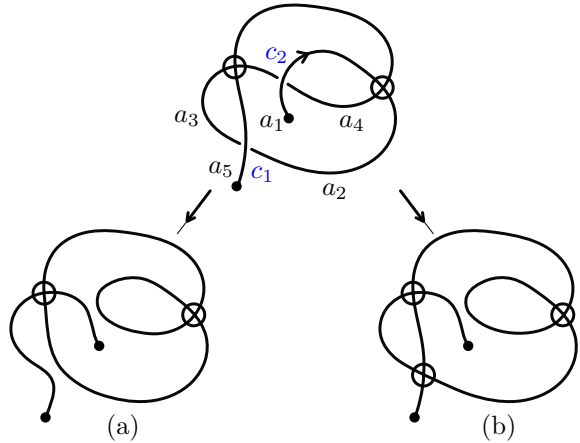
\begin{figure}[h!]
            \centering
            \begin{overpic}[width=0.5\linewidth]{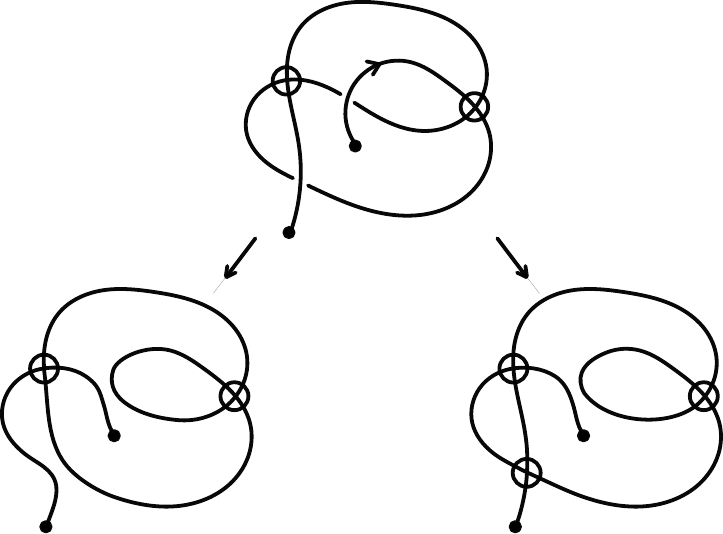}
                \put(44,52){$a_1$}
                \put(55,40){$a_2$}
                \put(29,53){$a_3$}
                \put(58,52){$a_4$}
                \put(35,44){$a_5$}
                \put(17,-2){(a)}
                \put(83,-2){(b)}
                \put(42.5,43){\color{blue}$c_1$}
                \put(45,64){\color{blue}$c_2$}
                
            \end{overpic}
            \vspace{0.5 cm}
            \caption{A labeled virtual knotoid diagram and its states.}
            \label{states}
        \end{figure}

    In order to obtain a state-sum invariant over states of $K_f$, each smoothing of a crossing utilized to obtain a state $S$ is associated with one of the mappings $A,B,V,C,D,U$ with respect to the orientation on $K_f$: If a crossing is a positive crossing then the vertical smoothing of the crossing is associated with $A(x,y) = A_{x,y}$, the horizontal smoothing of the crossing is associated with $B(x, y)= B_{x,y}$, and the virtual smoothing of the crossing is associated with $V(x,y)=V_{x,y}$, where $x,y$ denote the $X$- (input) colors of the crossing. If the crossing is a negative crossing, the mappings $C, D, U$ are utilized for vertical, horizontal and virtual smoothings, respectively. The rules are also depicted in Figure \ref{bvbskein}.

The verification that the $X$-virtual bracket defined for $K_f$ in the following definition remains unchanged under generalized Reidemeister moves is similar to the verification of the $X$-virtual bracket being an invariant of $X$-colored virtual knots and links. The readers are referred to \cite{nelson2019biquandle} for the details.   

\begin{definition}
Let $X$ be a biquandle, and $K_f$ be an $X$-colored virtual knotoid diagram.
The \textit{X-virtual bracket} invariant of an $X$-colored virtual knotoid diagram is defined as,

    $$\beta(K_f)=\sum_{S}\beta(S),$$
    where $\beta(S)$ is the product of the values of mappings associated with smoothings of crossings to obtain the state $S$ with $\delta^{m} \omega ^{-wr(K)}$ where $\delta \in R$, $\omega \in R^{x}$ are that appear in Conditions 1-2 in Definition \ref{defn:biquandlebracket}, $m$ is the number of components (up to virtual isotopy) in $S$, $wr(K)$ is the writhe of $K$.

    \end{definition}

    The  $X$-virtual bracket evaluation can be also described locally as in Figure \ref{bvbskein}, since the collection of states of a virtual knotoid diagram can be partitioned into states obtained by utilizing the vertical, horizontal, and the virtual smoothing at a chosen crossing. These local pictures can be considered as \textit{skein relations} of the state-sum.  

 \begin{figure}[h!]
    \centering
    \vspace{0.3 cm}
        \begin{overpic}[width= 0.7\linewidth]{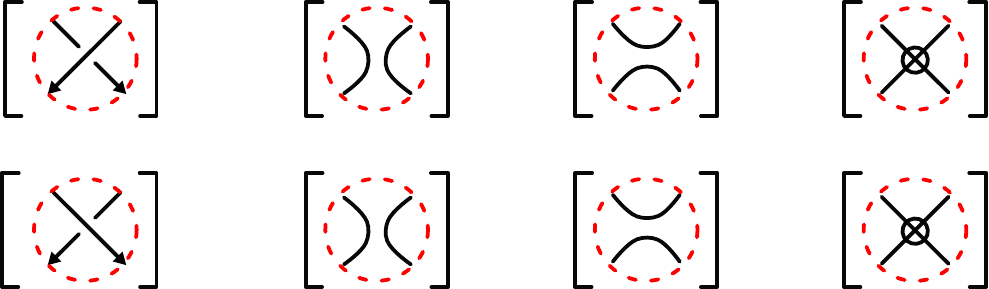}
       \put(2,26.5){$x$}
        \put(2.2,18.5){$y$}

        \put(2,9){$y$}
        \put(2,1){$x$}
        
        \put(18.5,23){= $A_{x,y}$}
        \put(47,23){$+B_{x,y}$}
        \put(74.7,23){$+V_{x,y}$}

        \put(18.5,5.5){= $C_{x,y}$}
        \put(47,5.5){$+D_{x,y}$}
        \put(74.7,5.5){$+U_{x,y}$}
    \end{overpic}
    \caption{Skein relations for $X$-virtual bracket.}
		\label{bvbskein}
    \end{figure}
    \begin{definition}

    We compute the $X$-virtual bracket value of all possible $X$-colorings of $K$ and form a multi-set of these values, which is denoted by $\Phi_X^{\beta, M}(K)$. More formally, 
    \[\Phi_X^{\beta, M}(K)=\{\beta(K_f)\;|\; f\in \text{Hom}(\mathcal{B}(K),X)\}.\]
    It is clear that this multi-set is an invariant of classical and virtual knotoids. We call this multi-set \textit{X- virtual bracket multi-set}. The cardinality of this multi-set is the biquandle counting invariant of $K$.

    \end{definition}

    \begin{definition}
        Let $K$ be a virtual knotoid diagram with the trivial fundamental biquandle coloring, denoted by $K_{1_{\mathcal{B}}}$ (see Example \ref{trivial}). The sum of the biquandle virtual bracket value contributions of the states of $K$, $\beta(K_{1_{\mathcal{B}}})$, is \textit{the fundamental biquandle virtual bracket of} $K$.
    \end{definition}

    Let $X$ be a finite biquandle and $\beta$ be an $X$-virtual bracket. Consider an $X$-coloring $f\in\text{Hom}(\mathcal{B}(K),X)$. The biquandle virtual bracket value $\beta(K_f)$ can be evaluated by replacing each coefficient $F_{a_i,a_j}$ in $\beta(K_{1_{\mathcal{B}}})$ with $F_{f(a_i),f(a_j)}$ of $\beta$ for all $F\in \{A,B,V,C,D,U\}$, and plugging the $\delta$ and $\omega$ values of $\beta^v$ in $\beta(K_{1_{\mathcal{B}}})$.
    
    \begin{example}
    \label{mbvb}
  In Figure \ref{states} we depict a virtual knotoid diagram $K$ that represents the virtual knotoid 2.1.1, 
  and the states of $K_{1_\mathcal{B}}$. For each state $S$  of $K_{1_\mathcal{B}}$, the biquandle virtual bracket $\beta(S) $ is given below $S$. Then the fundamental biquandle virtual bracket of $K$ is,
    \begin{align}
    \label{fbvb}
        &\beta(K_{1_\mathcal{B}})=\omega^{2}\,[\,\delta^2(D_{a_4,a_1}D_{a_3,a_4}+C_{a_4,a_1}U_{a_3,a_4}+U_{a_4,a_1}C_{a_3,a_4})\notag\\&+\delta(D_{a_4,a_1}C_{a_3,a_4}+D_{a_4,a_1}U_{a_3,a_4}+C_{a_4,a_1}D_{a_3,a_4}+C_{a_4,a_1}C_{a_3,a_4}+U_{a_4,a_1}D_{a_3,a_4}+U_{a_4,a_1}U_{a_3,a_4})].
    \end{align}

    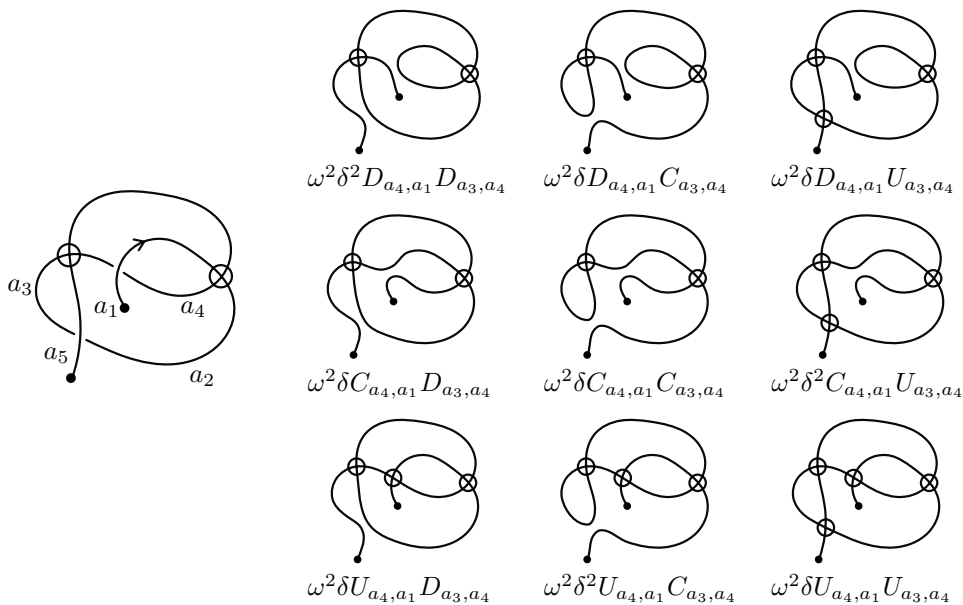
\begin{figure}[h!]
            \centering
            \begin{overpic}[width=0.8\linewidth]{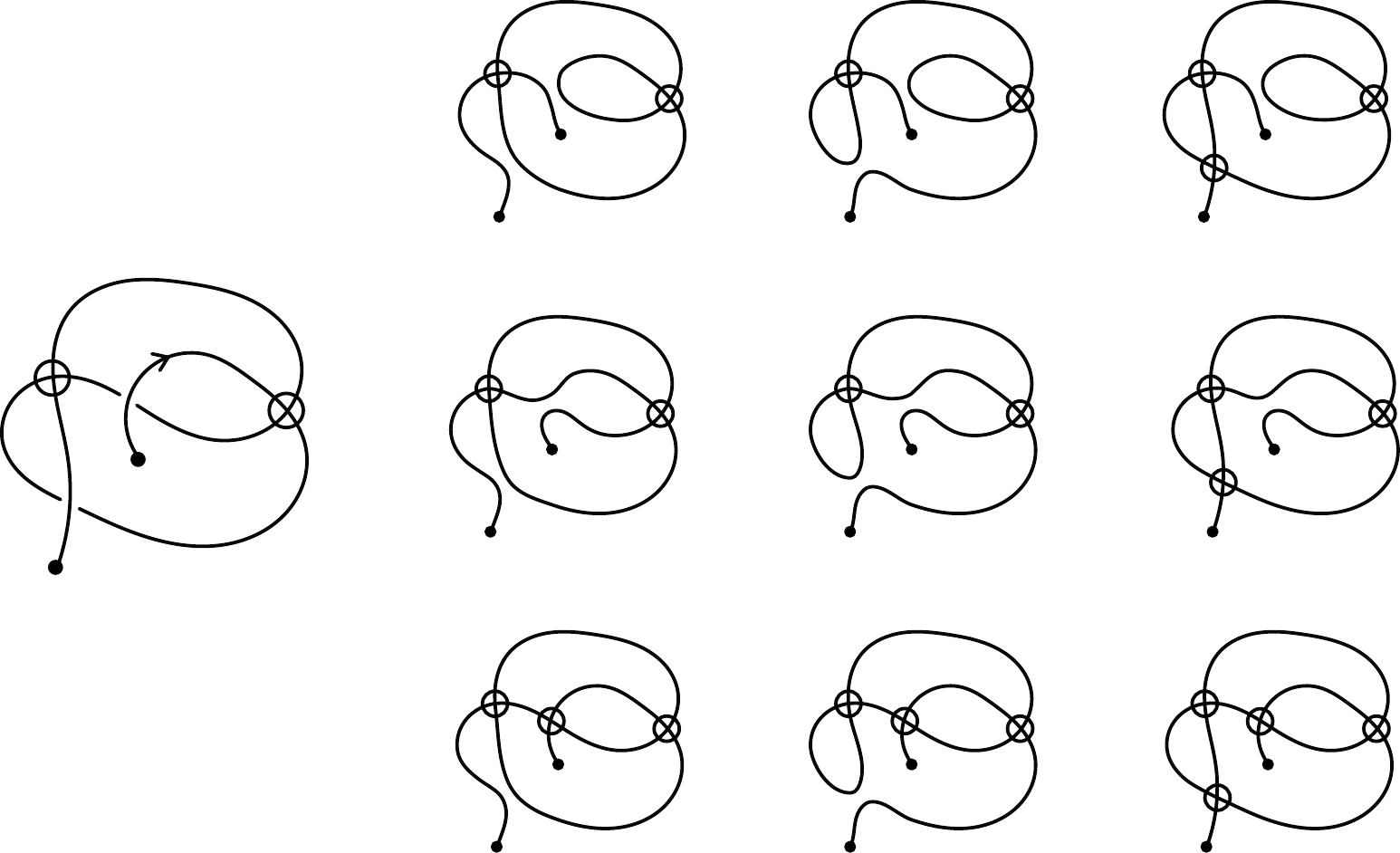}
                \put(6.5,27.5){$a_1$}
                \put(17,20){$a_2$}
                \put(-3,30){$a_3$}
                \put(16,27.5){$a_4$}
                \put(1,22.5){$a_5$}

                \put(30,41.5){$\omega^2\delta^2 D_{a_4,a_1}D_{a_3,a_4}$}
                \put(56,41.5){$\omega^2\delta D_{a_4,a_1}C_{a_3,a_4}$}
                \put(81,41.5){$\omega^2\delta D_{a_4,a_1}U_{a_3,a_4}$}

                \put(30,19){$\omega^2\delta C_{a_4,a_1}D_{a_3,a_4}$}
                \put(56,19){$\omega^2\delta C_{a_4,a_1}C_{a_3,a_4}$}
                \put(81,19){$\omega^2\delta^2 C_{a_4,a_1}U_{a_3,a_4}$}

                \put(30,-3.5){$\omega^2\delta U_{a_4,a_1}D_{a_3,a_4}$}
                \put(56,-3.5){$\omega^2\delta^2U_{a_4,a_1}C_{a_3,a_4}$}
                \put(81,-3.5){$\omega^2\delta U_{a_4,a_1}U_{a_3,a_4}$}
            \end{overpic}
            \vspace{0.5 cm}
            \caption{A labeled virtual knotoid diagram and its state contributions.}
            \label{states}
        \end{figure}
		Now, consider the Alexander biquandle over $X=\mathbb{Z}_3$ with the operations
        \[ x\un y = 2x+1 = x\ovr y,\]
	
	and let $R=\mathbb{Z}_5$. Our Python computations show that the following matrix of coefficient maps gives a biquandle virtual bracket,
	\[\left[\begin{array}{c c c | c c c | c c c | c c c | c c c | c c c}
		4 & 0 & 0 \;&\;1 & 0 & 0 \;&\;0 & 2 & 2 \;&\;4 & 0 & 0 \;&\;1 & 0 & 0 \;&\;0 & 3 & 3 \\
		0 & 4 & 0 \;&\;0 & 1 & 0 \;&\;4 & 0 & 1 \;&\;0 & 4 & 0 \;&\;0 & 1 & 0 \;&\;4 & 0 & 1 \\
		0 & 0 & 4 \;&\;0 & 0 & 1 \;&\;2 & 3 & 0 \;&\;0 & 0 & 4 \;&\;0 & 0 & 1 \;&\;3 & 2 & 0
	\end{array}\right],\]
	with $\delta=2$ and $\omega = 4$. 
    
    The biquandle homomorphism $f\in\text{Hom}(\mathcal{B}(K),X)$ defined as \[f(a_i)=\begin{cases}
        3, \quad\text{if $\,i\,$ is odd,}\\
        1, \quad\text{if $\,i\,$ is even,}
    \end{cases}\] for $i\in\{1,2,3,4,5\}$ is an $X$-coloring of $K$. In this case, the equation \eqref{fbvb} becomes 
    \begin{align}
    \label{fbvbc}
        &\beta(K_f)=\omega^{2}\,[\,\delta^2(D_{1,3}D_{3,1}+C_{1,3}U_{3,1}+U_{1,3}C_{3,1})\notag\\&+\delta(D_{1,3}C_{3,1}+D_{1,3}U_{3,1}+C_{1,3}D_{3,1}+C_{1,3}C_{3,1}+U_{1,3}D_{3,1}+U_{1,3}U_{3,1})].
    \end{align}
    Evaluating the Equation \eqref{fbvbc} for $\delta=2$, $\omega=4$, and taking coefficient matrices of the $X$-virtual bracket $\beta$ into account, we obtain
    $\beta(K_f)=3$. We repeat the same process for all $X$-colorings of $K$ and obtain the multi-set of $X$-virtual bracket values of $K$,
    \[\Phi_X^{\beta, M}(K)=\{3,3,2\}.\]
    We use our Python code to compute the multi-set of biquandle virtual bracket values of the virtual knotoid 4.1.1 and obtain $\Phi_X^{\beta,M}(4.1.1)=\{2,2,2\}$. Thus, the virtual knotoids 2.1.1 and 4.1.1 are not distinguished by the $X$-counting invariant, as it is equal to $3$ for both virtual knotoids, but by the multi-set of $X$-virtual bracket values. This shows that the multi-set of biquandle virtual bracket values is a proper enhancement of the biquandle counting invariant.
    \end{example}

    \begin{definition}
        Let $X$ be a finite biquandle, $\beta$ be an $X$-virtual bracket and $R$ be a number ring. The \textit{biquandle virtual bracket polynomial} of a virtual knotoid diagram $K$ is defined as
        \[\Phi_X^{\beta}(K)=\sum_{f\in \text{Hom}(\mathcal{B}(K),X)}u^{\beta(K_f)},\]
        where $u$ is a formal variable.
    \end{definition}
    
    \begin{example}
        It follows from Example \ref{mbvb} that, 
        \[\Phi_X^{\beta}(2.1.1)=2u^3+u^2,\qquad \Phi_X^{\beta}(4.1.1)=3u^2.\]
    \end{example}
    The biquandle counting matrix is an enhancement of the biquandle counting invariant for classical and virtual knotoids, as is pointed out in Section \ref{coloringinv}. In what follows, we enhance the biquandle virtual bracket value multi-set in a similar fashion. For a finite biquandle $X=\{x_1,x_2,...,x_n\}$ and for each $1\leq i,j\leq n$, we compute the biquandle virtual bracket values of $X_{ij}$-colored diagrams and obtain the multi-set of these values,
    
    \[\Phi_{X_{ij}}^{\beta, M}(K)=\{\beta(K_f)\;|\; f\in \text{Hom}_{ij}(\mathcal{B}(K),X)\}.\]
    
    \begin{theorem}
        The family of multi-sets $\{\Phi_{X_{ij}}^{\beta, M}(K)\}_{i,j\in X}$ is an invariant of virtual knotoids.
    \end{theorem}
    
    In the case where $R$ is a number ring, each multi-set in this family can be represented by its biquandle virtual bracket polynomial. Then, this family of multi-sets becomes a multi-set of polynomials which can be represented by a matrix as follows.
    
	\begin{definition}
        Consider a finite biquandle $X=\{x_1,x_2,...,x_n\}$, a virtual knotoid diagram $K$, a unitary commutative number ring $R$, and an $X$-virtual bracket $\beta$ over $R$. The \textit{X-virtual bracket matrix} of $K$ with respect to $\beta$ is defined as
        \[\mathcal{M}_X^{\beta}(K)=\left[\begin{array}{c c c}
             \beta_{11}^\upsilon&...&\beta_{1n}^\upsilon  \\
             \vdots& &\vdots\\
             \beta_{n1}^\upsilon&...&\beta_{nn}^\upsilon
        \end{array}\right]\]
        where \[\beta_{ij}^\upsilon=\sum_{f\in \text{Hom}_{ij}(\mathcal{B}(K),X)}u^{\beta(K_f)}.\]
    \end{definition}
    \begin{remark}
        By definition, $\{\Phi_{X_{ij}}^{\beta, M}(K)\}_{i,j\in X}$ is an enhancement of the multi-set of $X$-virtual bracket values $\Phi_{X}^{\beta, M}(K)$. In the case where $R$ is a number ring, the biquandle virtual bracket matrix is an enhancement of the biquandle virtual bracket polynomial. Moreover, since the cardinality of $\Phi_{X_{ij}}^{\beta, M}(K)$ is the value of the $(i,j)$-th entry in the biquandle counting matrix, the biquandle virtual bracket matrix is an enhancement of the biquandle counting matrix. The enhancement relations between these invariants of virtual knotoids are illustrated in the following diagrams. Left-hand side of the table represents the general case where $R$ is an arbitrary ring, and the right-hand side represents the case where $R$ is a number ring. An arrow from an invariant to another means that the latter is an enhancement of the former.
        \begin{figure}[h!]
        \begin{minipage}{0.45\textwidth}
        \[
        \begin{tikzcd}
         \Phi_{X}^{\mathbb{Z}}(K) \arrow[r] \arrow[rd] \arrow[d] & \Phi_{X}^{\beta,M}(K) \arrow[d] \\
        \mathcal{M}_X(K) \arrow[r] & \{\Phi_{X_{ij}}^{\beta,M}(K)\}_{i,j\in X}
        \end{tikzcd}
        \]
        \end{minipage}
        \hfill
        \begin{minipage}{0.45\textwidth}
        \[
        \begin{tikzcd}
         \Phi_{X}^{\mathbb{Z}}(K) \arrow[r] \arrow[rd] \arrow[d] & \Phi_{X}^{\beta}(K) \arrow[d] \\
        \mathcal{M}_X(K) \arrow[r] & \mathcal{M}_X^{\beta}(K)
        \end{tikzcd}
        \]
        \end{minipage}
    \end{figure}
    \end{remark}
    Next, we give some examples that show that these enhancements are proper.
    \begin{example}
    \label{bvb3133}
    Consider the virtual knotoid diagrams given in Figure \ref{3133}. Let $X$ be the biquandle given in Example \ref{mbvb}.
        \begin{figure}[h!]
        \begin{subfigure}{0.5\textwidth}
        \centering
            \includegraphics[width=30 mm]{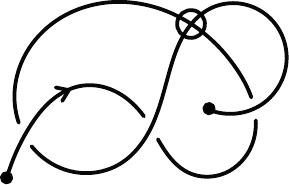}
            \caption{The virtual knotoid diagram 3.1.1.}
            \label{31}
        \end{subfigure}
        \begin{subfigure}{0.5\textwidth}
        \centering
            \includegraphics[width=25 mm]{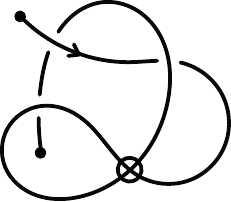}
            \caption{The virtual knotoid diagram 3.1.3.}
            \label{33}
        \end{subfigure}
        \caption{}
        \label{3133}
    \end{figure}

    The $X$-coloring matrices of these virtual knotoids are the same,
    \[\mathcal{M}_X(3.1.1)=\left[\begin{array}{c c c}
		1 & 0 & 0 \\
        0 & 1 & 0 \\
        0 & 0 & 1
	\end{array}\right]= \mathcal{M}_X(3.1.3).\]
    Consider the biquandle virtual bracket given in Example \ref{mbvb}. Our Python computations show that the $X$-virtual bracket matrices of these virtual knotoid diagrams are as follows.
    \[\mathcal{M}_X^{\beta}(3.1.1)=\left[\begin{array}{c c c}
		u^2 & 0 & 0 \\
        0 & u^2 & 0 \\
        0 & 0 & u^2
	\end{array}\right],\qquad \mathcal{M}_X^{\beta}(3.1.3)=\left[\begin{array}{c c c}
		u^3 & 0 & 0 \\
        0 & u^3 & 0 \\
        0 & 0 & u^2
	\end{array}\right].\]
    \end{example}
    Example \ref{bvb3133} shows that the biquandle virtual bracket matrix is a stronger invariant than the biquandle counting matrix for virtual knotoids.
\begin{example}
\label{comp}
    Consider the biquandle $X=\mathbb{Z}_3$ with operations defined as
    \[x\un y=x+1=x\ovr y\]
    for all $x,y\in X$.
    We use our Python code to generate the $X$-virtual bracket $\beta$ over $\mathbb{Z}_{37}$ with coefficient matrix
    \[\left[\begin{array}{c c c | c c c | c c c | c c c | c c c | c c c}
		7 & 0 & 0 \;&\;11 & 0 & 0 \;&\;0 & 6 & 5 \;&\;16 & 0 & 0 \;&\;27 & 0 & 0 \;&\;0 & 2 & 9 \\
		0 & 7 & 0 \;&\;0 & 11 & 0 \;&\;10 & 0 & 8 \;&\;0 & 16 & 0 \;&\;0 & 27 & 0 \;&\;10 & 0 & 7 \\
		0 & 0 & 7 \;&\;0 & 0 & 11 \;&\;1 & 2 & 0 \;&\;0 & 0 & 16 \;&\;0 & 0 & 27 \;&\;1 & 6 & 0
	\end{array}\right],\]
    and with $\delta=5$, $\omega=9$. 
    \begin{figure}[h!]
        \begin{subfigure}{0.5\textwidth}
        \centering
            \includegraphics[width=25 mm]{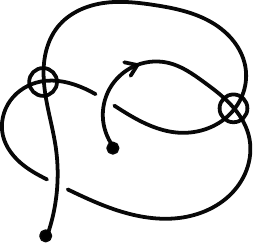}
            \caption{2.1.1}
            \label{v21}
        \end{subfigure}
        \begin{subfigure}{0.5\textwidth}
        \centering
            \includegraphics[width=31 mm]{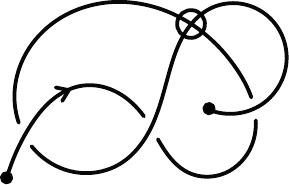}
            \caption{3.1.1}
            \label{v31}
        \end{subfigure}
        \caption{}
    \end{figure}
    Consider the virtual knotoids 2.1.1 and 3.1.1. We compute the values of the invariants mentioned above with respect to the biquandle $X$ and the $X$-virtual bracket $\beta$ for these two virtual knotoids using our Python code and obtain Table \ref{values}.

    \begin{table}[h!]
\centering
\[
\begin{tabular}{!{\vrule width 1.1pt}>{\centering\arraybackslash}m{2cm}!{\vrule width 1.1pt}>{\centering\arraybackslash}m{3cm}|>{\centering\arraybackslash}m{3cm}!{\vrule width 1.1pt}}
\specialrule{1.1pt}{0pt}{0pt}
 \rule{0pt}{13pt}& 2.1.1 & 3.1.1\\
\specialrule{1.1pt}{0pt}{0pt}
\rule{0pt}{16pt}$\Phi_X^{\mathbb{Z}}$ &3 &3 \\
\hline
\rule{0pt}{16pt}$\Phi_X^{\beta}$ &$u^{19}+u^{27}+u^{34}$ &$u^{19}+u^{27}+u^{34}$ \\
\hline
\rule{0pt}{40pt}$\mathcal{M}_X$&$\begin{bmatrix}
1 & 0 & 0 \\
0 & 1 & 0 \\
0 & 0 & 1
\end{bmatrix}$&$\begin{bmatrix}
1 & 0 & 0 \\
0 & 1 & 0 \\
0 & 0 & 1
\end{bmatrix}$\\[2.5em]
\hline
\rule{0pt}{40pt}$\mathcal{M}_X^{\beta}$&$\begin{bmatrix}
u^{27} & 0 & 0 \\
0 & u^{19} & 0 \\
0 & 0 & u^{34}
\end{bmatrix}$&
$\begin{bmatrix}
u^{19} & 0 & 0 \\
0 & u^{34} & 0 \\
0 & 0 & u^{27}
\end{bmatrix}$\\[2.5em]
\specialrule{1.1pt}{0pt}{0pt}

\end{tabular}
\]
\caption{Biquandle counting values and biquandle virtual bracket values of the virtual knotoids 2.1.1 and 3.1.1.}
\label{values}
\end{table}
It is seen in the table that the biquandle counting invariant, the biquandle counting matrix invariant and the biquandle virtual bracket polynomial invariant all fail to distinguish the virtual knotoids 2.1.1 and 3.1.1. However, the biquandle virtual bracket matrix invariant distinguishes these two virtual knotoids. This shows that the biquandle virtual bracket matrix is a proper enhancement of these three invariants of virtual knotoids.
    
    We compute the $X$-virtual bracket matrices of some of the virtual knotoids and obtain Table \ref{table}. Observe that the virtual knotoids in the table are not distinguished by the $X$-counting invariant or the $X$-counting matrix and more pairs of virtual knotoids as in Example \ref{comp} can be found in Table \ref{table}.
\end{example}
\begin{table}[h!]
\centering
\[
\begin{tabular}{c|m{4.5cm}||c|m{3.5cm}}
\toprule
$\Phi_X^{\beta}(K)$ & $K$ & $\Phi_X^{\beta}(K)$ & $K$ \\
\midrule
$\begin{bmatrix}
u^{27} & 0 & 0 \\
0 & u^{19} & 0 \\
0 & 0 & u^{34}
\end{bmatrix}$
&
2.1.1
&
$\begin{bmatrix}
u^{27} & 0 & 0 \\
0 & u^{21} & 0 \\
0 & 0 & u^{14}
\end{bmatrix}$
&
2.1.2
\\[2em]
$\begin{bmatrix}
u^{19} & 0 & 0 \\
0 & u^{34} & 0 \\
0 & 0 & u^{27}
\end{bmatrix}$
& 3.1.1, 3.1.4
&
$\begin{bmatrix}
u^{27} & 0 & 0 \\
0 & u^{19} & 0 \\
0 & 0 & u^{34}
\end{bmatrix}$
& 3.1.2, 3.1.3
\\[2em]

$\begin{bmatrix}
u^7 & 0 & 0 \\
0 & u^7 & 0 \\
0 & 0 & u^7
\end{bmatrix}$
& 3.1.5, 3.1.10
&
$\begin{bmatrix}
u^5 & 0 & 0 \\
0 & u^{32} & 0 \\
0 & 0 & u^5
\end{bmatrix}$
& 3.1.6
\\[2em]

$\begin{bmatrix}
u^5 & 0 & 0 \\
0 & u^5 & 0 \\
0 & 0 & u^{32}
\end{bmatrix}$
& 3.1.7
&
$\begin{bmatrix}
u^{34} & 0 & 0 \\
0 & u^{27} & 0 \\
0 & 0 & u^{19}
\end{bmatrix}$
& 3.1.8
\\[2em]

$\begin{bmatrix}
u^{32} & 0 & 0 \\
0 & u^5 & 0 \\
0 & 0 & u^5
\end{bmatrix}$
& 3.1.9
&
$\begin{bmatrix}
u^6 & 0 & 0 \\
0 & u^{36} & 0 \\
0 & 0 & u^{33}
\end{bmatrix}$
& 4.1.1
\\[2em]

$\begin{bmatrix}
u^{28} & 0 & 0 \\
0 & u^{4} & 0 \\
0 & 0 & u^{33}
\end{bmatrix}$
& 4.1.2
&
$\begin{bmatrix}
u^{30} & 0 & 0 \\
0 & u^{3} & 0 \\
0 & 0 & u^{6}
\end{bmatrix}$
& 4.1.3
\\[2em]

$\begin{bmatrix}
u^{27} & 0 & 0 \\
0 & u^{22} & 0 \\
0 & 0 & u^{30}
\end{bmatrix}$
& 4.1.4
&
$\begin{bmatrix}
u & 0 & 0 \\
0 & u^{32} & 0 \\
0 & 0 & u^{19}
\end{bmatrix}$
& 4.1.5
\\[2em]

$\begin{bmatrix}
u^{24} & 0 & 0 \\
0 & u^{20} & 0 \\
0 & 0 & u^{13}
\end{bmatrix}$
& 4.1.6
&
$\begin{bmatrix}
u^{31} & 0 & 0 \\
0 & u^{13} & 0 \\
0 & 0 & u
\end{bmatrix}$
& 4.1.7
\\[2em]

$\begin{bmatrix}
u^{36} & 0 & 0 \\
0 & u^{33} & 0 \\
0 & 0 & u^{6}
\end{bmatrix}$
& 4.1.8
&
$\begin{bmatrix}
u^{20} & 0 & 0 \\
0 & u^{17} & 0 \\
0 & 0 & u^{19}
\end{bmatrix}$
& 4.1.9
\\[2em]

$\begin{bmatrix}
u^{22} & 0 & 0 \\
0 & u^{30} & 0 \\
0 & 0 & u^{27}
\end{bmatrix}$
& 4.1.10
&
$\begin{bmatrix}
u^{19} & 0 & 0 \\
0 & u^{20} & 0 \\
0 & 0 & u^{17}
\end{bmatrix}$
& 5.1.1
\\[2em]

$\begin{bmatrix}
u^{24} & 0 & 0 \\
0 & u^{20} & 0 \\
0 & 0 & u^{13}
\end{bmatrix}$
& 5.1.2
&
$\begin{bmatrix}
u^{30} & 0 & 0 \\
0 & u^{3} & 0 \\
0 & 0 & u^{6}
\end{bmatrix}$
& 5.1.3
\\[2em]

\end{tabular}
\]
\end{table}
\begin{table}[h!]
\centering
\[
\begin{tabular}{c|m{4.5cm}||c|m{3.5cm}}

$\begin{bmatrix}
u^{19} & 0 & 0 \\
0 & u^{27} & 0 \\
0 & 0 & u^{17}
\end{bmatrix}$
&
5.1.4
&
$\begin{bmatrix}
u^{12} & 0 & 0 \\
0 & u^{3} & 0 \\
0 & 0 & u^{35}
\end{bmatrix}$
&
5.1.5
\\[2em]
$\begin{bmatrix}
u^{3} & 0 & 0 \\
0 & u^{23} & 0 \\
0 & 0 & u^{28}
\end{bmatrix}$
& 5.1.6
&
$\begin{bmatrix}
u^{35} & 0 & 0 \\
0 & u^{19} & 0 \\
0 & 0 & u^{25}
\end{bmatrix}$
& 5.1.7
\\[2em]

$\begin{bmatrix}
u^{20} & 0 & 0 \\
0 & u^{13} & 0 \\
0 & 0 & u^{24}
\end{bmatrix}$
& 5.1.8
&
$\begin{bmatrix}
u^{27} & 0 & 0 \\
0 & u^{31} & 0 \\
0 & 0 & u^{17}
\end{bmatrix}$
& 5.1.9
\\[2em]

$\begin{bmatrix}
u^{11} & 0 & 0 \\
0 & u^{15} & 0 \\
0 & 0 & u^{32}
\end{bmatrix}$
& 5.1.10
&

& 
\\
\bottomrule

\end{tabular}
\]
\caption{A table of biquandle virtual bracket matrices of virtual knotoids.}
\label{table}
\end{table}
\section{Discussion}

The biquandle virtual brackets are an infinite family of invariants for virtual knotoids as they are built on a choice of a biquandle $X$, a ground ring $R$, and a coefficient matrix. The computability and the strength of this invariant can be adjusted by these choices. Working even on a slightly large ring $R=\mathbb{Z}_{37}$, we obtain Table \ref{table}, where all virtual knotoids presented are distinguished except for three couples. However, as pointed out in \cite{nelson2019biquandle}, it is a difficult task to determine the coefficient maps $A,B,V,C,D,U$ satisfying the given axioms of a biquandle virtual bracket. In our calculations, we use some restrictions to make this task bearable, such as working with very small biquandles and picking the ground ring $R$ as a field so that the axioms are eased as the images of the coefficient maps are all invertible. 
Without these restrictions, we need efficient algorithms to be able to find biquandle virtual brackets over larger ground rings. It can be an interesting task to use these biquandle virtual brackets to extend Bartholomew's table of virtual knotoids in \cite{virk2}.

\clearpage
\bibliography{references}
\bibliographystyle{plain}
	
\end{document}